\documentclass[a4paper,10pt]{article}
\usepackage{amsfonts,amsmath,amssymb}
\usepackage{amsthm}
\usepackage{geometry}
\usepackage{graphicx}
\usepackage{hyperref} 
\usepackage{mathrsfs}  
\newtheorem{theorem}{Theorem}[section]
\newtheorem{lemma}{Lemma}[section]
\newtheorem{corollary}{Corollary}[section]
\newtheorem{conjecture}{Conjecture}[section]

\title{The Ramsey Number for Tree Versus Wheel with Odd Order} 

\author
{Yusuf Hafidh and Edy Tri Baskoro\footnote{corresponding author}\\
\\
\normalsize{Combinatorial Mathematics Research Group}\\
\normalsize{Faculty of Mathematics and Natural Sciences}\\
\normalsize{Institut Teknologi Bandung}\\
\normalsize{Jalan Ganesa 10 Bandung, Indonesia}\\
\\
\normalsize{Emails: yusufhafidh@gmail.com, ebaskoro@math.itb.ac.id}\\
}

\date{}

\begin{document}

\maketitle 


\begin{abstract}
Chen et al. (2004) strongly conjectured that $R(T_n,W_m)=2n-1$ if the maximum degree of $T_n$ is small and $m$ is even. Related to the conjecture, it is interesting to know for which tree $T_n$, we have $R(T_n,W_m)>2n-1$ for even $m$.
In this paper, we find the Ramsey number $R(T_n,W_8)$ for tree $T_n$ with the maximum degree of $T_n$ is at least $n-3$, namely $R(T_n,W_8)>2n-1$ for almost all such tree $T_n$.
We also prove that if the maximum degree of $T_n$ is large, then $R(T_n,W_m)$ is a function of both $m$ and $n$. In the end, we refine the conjecture of Chen et al. by giving the condition of small maximum degree.


\end{abstract} 


\section{Introduction}

Let $G$ be a graph with the vertex set $V(G)$ and edge set $E(G)$. All graphs in this paper are assumed to be finite, undirected, and simple graphs. The {\em order} of $G$ is the number of vertices of $G$. The graph $\overline{G}$, the
{\em complement} of graph $G$, is obtained from the complete graph on $|G|$ vertices by deleting the edges of $G$. Let $U\subseteq V(G)$, $G[U]$ is the {\em induced subgraph} of $G$ by $U$; that is the maximal subgraph of $G$ with the vertex set $U$. For $v\in V(G)$ and $U,W\subseteq V(G)$, we denote $N_U(v)=\{u\in U : uv\in E(G) \}$, and $E(U,W)$ as the edges between $U$ and $W$. The degree of a vertex $v\in V(G)$ is denoted by $\deg(v)$. The maximum (minimum) degree of all vertices in $G$ is denoted by $\Delta(G)$ ($\delta(G)$), respectively.

Let $K_n$ be a {\em complete graph} on $n$ vertices and 
let $C_n$ be a {\em cycle} on $n$ vertices.
Let $T_n$ be a {\em tree}, namely a connected graph with no cycle, on $n$ vertices. A special
tree called a {\em star} is denoted by $S_n$, namely a tree on $n$ vertices with maximum degree $n-1$. Let $W_m$ be a wheel on $m+1$ vertices that consists of a cycle $C_m$ with one additional vertex being adjacent to all vertices of $C_m$. We use $kG$ to denote the disjoint union of $k$ copies of $G$. 
By using the same notation in \cite{CZZ04}, denote by $S_n(l,m)$ a tree of order $n$ obtained from $S_{n-m\times l}$ by subdividing each of $l$ chosen edges $m$ times, and 
denote by $S_n(l)$ a tree of order $n$ obtained from a star $S_l$ and $S_{n-l}$ by adding an edge joining their centers.

For two graphs $G$ and $H$, a graph $F$ is called $(G,H)$-{\em good} if $F$ contains no $G$ and $\overline{F}$ contains no $H$. Any $(G,H)$-good graph on $n$ vertices is called a $(G,H,n)$-{\em good}. The \textit{Ramsey number} $R(G,H)$ is defined as the smallest positive integer $n$ such that no $(G,H,n)$-good graph exists.

Chv\'{a}tal and Harary (1972) studied the Ramsey numbers for graphs and established the lower bound: $R(G,H)\geq(\chi(H)-1)(c(G)-1)+1$,
where $c(G)$ is the number of vertices of the largest component of $G$, and $\chi(H)$ is the chromatic number of $H$. If $G=T_n$, $H=W_m$ and $n \geq m \geq 3$, we obtain
$R(T_n,W_m)\geq 2n-1$ for even $m$, and $R(T_n,W_m)\geq 3n-2$ for odd $m$.

The computation of the Ramsey number $R(T_n,W_m)$ has been extensively investigated. However, the results are still far from satisfied. For stars $S_n$, the value of $R(S_n,W_m)$ is not always equal to the Chv\'{a}tal-Harary bound \cite{SB01,CZZ04a,HBA05,ZCZ08,ZCC12}, but for trees $T_n$ other than a star and $S_n(1,2)$, all the known Ramsey numbers $R(T_n,W_m)$ are equal to the Chv\'{a}tal-Harary bound \cite{BSM02,CZZ04,CZZ05a,CZZ05b,CZZ06,ZBC16}.

Y. Chen, Y. Zhang, and K. Zhang (2004) strongly conjectured that $R(T_n,W_m)=2n-1$ if the maximum degree of $T_n$ is small and $m$ is even. For a tree $T_n$ with large maximum degree and even $m$, the $R(T_n,W_m)$ is also unknown in general. 
In this paper, we shall determine the Ramsey number $R(T_n,W_8)$ for all trees $T_n$ of order $n$ with the maximum degree of $T_n$ is at least $n-3$.
We also prove that if the maximum degree of $T_n$ is large, then $R(T_n,W_m)$ is a function of both $m$ and $n$. In the end, we refine the conjecture by giving the condition of small maximum degree.

\section{The Ramsey Number $R(T_n,W_m)$}

Related to the revised conjecture proposed by Chen et al. in \cite{CZZ04}, it is interesting to know for which tree $T_n$ we have $R(T_n,W_m)>2n-1$ for even $m$. The first main result deals with the Ramsey number $R(T_n,W_8)$ for all trees $T_n$ with $\Delta(T_n)\geq n-3$ other than a star, as stated in Theorem \ref{RTnW8}. In the second main result, we derive a general lower bound for $R(T_n,W_m)$ for trees $T_n$ with large maximum degree and even $m$. This lower bound is a function of both $m$ and $n$. The third main result discusses the conjecture given by Chen et al., we refine this conjecture by providing the condition of small maximum degree for tree $T_n$.

Trees $T_n$ with $\Delta(T_n)\geq n-3$ will be isomorphic to either $S_n$, $S_n(1,1)$, $S_n(1,2)$, $S_n(2,1)$, or $S_n(3)$.

\begin{theorem}\label{RTnW8} The Ramsey number $R(T_n,W_8)$ for $\Delta(T_n)\geq n-3$ is given below.
\begin{itemize}
	\item $	R(S_n(1,1),W_8)=\begin{cases}
	2n\phantom{-11},\ n\equiv 0\ (\text{mod}\ 2)\\
	2n+1,\ n\equiv 1\ (\text{mod}\ 2)
	\end{cases}; n\geq 5 $
	\item $R(S_n(1,2),W_8)=\begin{cases}
	2n\phantom{-11},\ n\not\equiv 3\ (\text{mod}\ 4)\\
	2n+1,\ n\equiv 3\ (\text{mod}\ 4)
	\end{cases}; n\geq 8 $
	\item $R(S_n(2,1),W_8)=\begin{cases}
	2n-1,\ n\equiv 1\ (\text{mod}\ 2)\\
	2n\phantom{-11},\ n\equiv 0\ (\text{mod}\ 2)
	\end{cases}; n\geq 8 $
	\item $R(S_n(3),W_8)\phantom{,1}=\begin{cases}
	2n-1,\ n\equiv 1\ (\text{mod}\ 2)\\
	2n\phantom{-11},\ n\equiv 0\ (\text{mod}\ 2)
	\end{cases}; n\geq 8 $
\end{itemize}
\end{theorem}
The proof of Theorem \ref{RTnW8} is given in the last section. For $R(S_n,W_8)$, see Theorems \ref{RSnW8even} and \ref{RSnW8odd}.

It has been showed in \cite{CZZ04}, $R(S_n(1,1),W_m)$ is a function of both $m$ and $n$, for {\bf even} m. Precisely, $R(S_n(1,1),W_m)\geq 2n+m/2-3$ for even $m$ and $n=km/2+3$ for any integer $k\geq2$.
It is not a big surprise that the values of $R(S_n(1,2),W_m)$, $R(S_n(2,1),W_m)$, and $R(S_n(3),W_m)$ depends on both $m$ and $n$. 

For even $m\geq 8$, and $n\equiv 4(\text{mod}\ m/2)$, define $G=\overline{H}\cup K_{n-1}$ with $H=(\frac{2n+m-8}{m})K_{m/2}$. Then, $G$ is a $(T,W_m,2n+m/2-5)$-good graph for any tree $T$ with $\Delta(T)\geq n-3$. Hence we have the following theorem.

\begin{theorem}\label{LB}
	For trees $T_n$ with $\Delta(T_n)=n-3$, $n\equiv 4 \;(\text{mod}\ m/2)$ and even $m\geq8$, $R(T_n,W_m)\geq 2n+m/2-4$. $\square$
\end{theorem}

Moreover, we could modify the graph $H$ above to obtain a lower bound for the other values of $n$. The modification of $H$ could be done by keeping certain properties such as the graph $H$ is $(n-4)$-regular and order $n-4+m/2$. 
The graph $H$ is also an union of some graphs with each has order less than $m$. 
For instance, if $n\equiv 2 \;(\text{mod}\ m/2)$, 
we could define $H=K_{m/2-1,m/2-1}\cup(\frac{2n-4}{m})K_{m/2}$. 
In this case, we can easily verify the following theorem.
\begin{theorem}
	For trees $T_n$ with $\Delta(T_n)=n-3$, $n\equiv 2 \;(\text{mod}\ m/2)$ and even $m\geq8$, $R(T_n,W_m)\geq 2n+m/2-4$. $\square$
\end{theorem}

For even $m$ and $1\leq t\leq m/2-2$, we can get a lower bound for $R(S_n(1,2t),W_m)$ by defining $G=\overline{H}\cup K_{n-1}$ with $H=(\frac{2n+m-2t-4}{m})K_{m/2}$ for $n\equiv t+2 \;(\text{mod}\ m/2)$. Since $G$ is a $(S_n(1,2t),W_m,2n+m/2-t-3)$-good graph, we have the following theorem.
\begin{theorem}\label{RSn2t}
	For $n\equiv t+2 \;(\text{mod}\ m/2)$ and even $m\geq8$, $R(S_n(1,2t),W_m)\geq 2n+m/2-t-2$. $\square$
\end{theorem}
For $t=1$, Theorem \ref{RSn2t} gives a better lower bound for $R(S_n(1,2),W_m)$ compared to Theorem \ref{LB}.
\begin{corollary}
	For $n\equiv 3 \;(\text{mod}\ m/2)$ and even $m\geq8$, $R(S_n(1,2t),W_m)\geq 2n+m/2-3$.
\end{corollary}

In \cite{BSM02}, it is conjectured that for every tree $T_n$ other than a star and $n>m$, the Ramsey number $R(T_n,W_m)$ is equal to the Chv\'{a}tal-Harary bound ($R(T_n,W_m)=2n-1$ for even $m \geq 6$, and $R(T_n,W_m)=3n-2$ for odd $m\geq 7$).
For odd $m$, the conjecture is true for $n\geq \frac{m+1}{2}\geq3$ \cite{HBA05} and for all ES-Trees which is believed to be all Trees \cite{ZBC16}. However, the conjecture must be refined for even $m$, see Theorem \ref{RTnW8} and \cite{CZZ04}.
In 2004, Chen et al. believed that the conjecture is true for even $m$ if $T_n$ has small maximum degree. However they did not specify the condition of "small" maximum degree. In this paper, we provide a refinement of the conjecture by giving the condition of small maximum degree as follows.

\begin{conjecture}
	For even $m$ and $n> m\geq 4$,
	\[\Delta(T_n)\leq n-m+2\quad \Rightarrow\quad R(T_n,W_m)=2n-1.\]
\end{conjecture}
The conjecture is true for $m=4$ (\cite{BSM02}) and $m=6$ (\cite{CZZ05a,CZZ05b,CZZ06}). We cannot increase the condition of the maximum degree because Theorem \ref{RSn2t} gives the following corollary.

\begin{corollary}
	For $n\equiv 0 \;(\text{mod}\ m/2)$ and even $m\geq8$, $\Delta(S_n(1,m-4))= n-m+3$ and $R(S_n(1,m-4),W_m)\geq 2n$. $\square$
\end{corollary}

\section{Proof of Theorem \ref{RTnW8}}

We give some known Ramsey number and lemmas we use to prove Theorem $\ref{RTnW8}$.
\begin{theorem}\cite{ZCZ08}\label{RSnW8even}
	$R(S_n,W_8)=2n+2,\ \text{for even}\ n\geq6$.
\end{theorem}
\begin{theorem}\cite{ZCC12}\label{RSnW8odd}
	$R(S_n,W_8)=2n+1,\ \text{for odd}\ n\geq5$.
\end{theorem}
\begin{lemma} \cite{CZZ04}
	\label{Lm1}
	Let $G$ be a graph of order $n\geq6$ with $\delta(G)\geq n-3$, then $G$ contains $S_n(3)$ and $S_n(2,1)$.
\end{lemma}

\begin{lemma}
	\label{Lm2}
	Let $G$ be a graph of order $2n$, $n\geq8$. 
	If $G$ contains $S_n(1,1)$ and $\overline{G}$ contains no $W_8$, 
	then $G$ must contain $S_n(1,2)$, $S_n(2,1)$ and $S_n(3)$.
\end{lemma}

\begin{proof} 
	Let $G$ be a graph satisfying the above assumptions. Let $\{u_0,u_1,\cdots,u_{n-1}\}$ be the set of the vertices
	of $S_n(1,1)$ in $G$ with $u_0$ as the hub and $u_0 u_1,u_1 u_{n-1}\in E(G)$. 
	Let $U=\{u_2,u_3,\cdots,u_{n-2}\}$ and $W=V(G)-V(S_n(1,1))=\{w_1,w_2,\cdots,w_n\}$.
	
	First, we prove that $G$ contains $S_n(2,1)$. Assume $G$ contains no $S_n(2,1)$ then $E(U,W)=\emptyset$ and $E(G[U])=\emptyset$. If $\delta(G[W])\geq n-3$ then by Lemma \ref{Lm1}, $W$ contains a $S_n(2,1)$, therefore $\delta(G[W])\leq n-4$ and $\Delta(\overline{G}[W])\geq 3$. Let $w_1\in W$ with $\{w_2,w_3,w_4\}\subseteq N_{\overline{G[W]}}(w_1)$. This implies that the subgraph of $\overline{G}$ induced by 
	$\{w_1,w_2,w_3,w_4,u_2,u_3,u_4,u_5,$ $u_6\}$ contains  a $W_8$ with the hub $w_1$  and $w_2u_2w_3u_3w_4u_4u_5u_6w_2$ as the cycle, a contradiction.
	
	Next we prove that $G$ contains $S_n(3)$. Assume $G$ contains no $S_n(3)$ then 
	$N(u_1)\subseteq\{u_0,u_{n-1}\}$ and $|N_W(u_i)|\leq 1$ for $i \in [2,n-2]$. 
	Since $E(U,W)\leq n-3$ and $|W|=n\geq8$, there exist $w_1,w_2,w_3,w_4$ with $N_U(w_i)\leq1$ for $i=1,2,3,4$. 
	This means that the subgraph of $\overline{G}$ induced by $\{u_1,u_2,u_3,u_4,u_5,w_1,w_2,w_3,w_4\}$ 
	contains a $W_8$ with $u_1$ as the hub, a contradiction.
	
	Finally we prove that $G$ contains $S_n(1,2)$. 
	Assume $G$ contains no $S_n(1,2)$, then 
	$N(u_{n-1})\subseteq\{u_0,u_1\}$ and 
	$|N_U(w)|\leq1$ for each vertex $w\in W$. 
	If there is a vertex in $U$ with at least $3$ neighbors in $W$, say $\{w_1,w_2,w_3\}\subseteq{N_W(u_2)}$, then the subgraph of $\overline{G}$ induced by 
	$\{u_3,u_4,u_5,u_6,u_{n-1},$ $w_1,w_2,w_3,w_4\}$ contains a $W_8$ with $u_{n-1}$ as the hub, a contradiction. If there is no vertex in $U$ with at least $3$ neighbors in $W$ then any four vertices $u_2,u_3,u_4,u_5$ in $U$ and any four vertices $w_1,w_2,w_3,w_4$ in $W$ together with $u_{n-1}$ will induce a $S_n(1,2)$ in $\overline{G}$ with $u_{n-1}$ as the hub, a contradiction. 
\end{proof}

\begin{lemma}
	\label{Lm3}
	If $G$ be a graph of order $n\geq9$ with $\delta(G)\geq n-4$ and $\Delta(G)\geq n-3$, then $G$ contains $S_n(3)$ and $S_n(2,1)$.
\end{lemma}

\begin{proof}
	Let $V(G)=\{w_1,w_2,\cdots,w_n\}$ and $w_1$ be the vertex with degree at least $n-3$ and $w_4,w_5,\cdots,w_n\in N(w_1)$. First, we prove that $G$ contains a $S_n(3)$. 
	If $w_2,w_3\notin N(w_1)$ or $w_2w_3\notin E(G)$, then since $\delta(G)\geq n-4$ and $n\geq9$, $|N(w_1)\cap N(w_2)|+|N(w_1)\cap N(w_3)|\geq n-5 + n-5 \geq n-2$ which implies $N(w_1)\cap N(w_2)\cap N(w_3)\ne\emptyset$, and hence $G$ contains a $S_n(3)$. Otherwise, without loss of generality $w_2\in N(w_1)$ and $w_2w_3\in E(G)$. Since $\delta(G)\geq n-4$ and $n \geq 9$, we have $\deg(w_2)\geq5$ which implies $N(w_1)\cap N(w_2)\ne\emptyset$, and hence $G$ contains a $S_n(3)$.
	
	Next, we prove $G$ contains a $S_n(2,1)$. Since $\delta(G)\geq n-4$ and $n \geq 9$, we have $\deg(w_i)\geq5$ for all $i$. Therefore $|N(w_1)\cap N(w_2)|\geq3$ and $|N(w_1)\cap N(w_3)|\geq3$. It is not hard to see that $G$ contains a $S_n(2,1)$.
\end{proof}

\begin{center}
\subsection*{Proof of Theorem \ref{RTnW8}}
\end{center}

We will divide Theorem $\ref{RTnW8}$ into smaller theorems and prove each of them separately.

\begin{theorem}
	\label{Th1}
	$R(S_n(1,1),W_8)=2n+1$ for odd $n \geq 5$.
\end{theorem}
\begin{proof} Let us consider the graph $G=\overline{(\frac{n+1}{4})K_4}\cup K_{n-1}$, 
	for $n\equiv 3(\text{mod}\ 4)$ and $G=\overline{K_{3,3}\cup(\frac{n-5}{4})K_4}\cup K_{n-1}$, for $n\equiv 1(\text{mod}\ 4)$. It is obvious that $G$ contains no $S_n(1,1)$ and $\overline{G}$ contains no $W_8$. 
	Hence, we have that $R(S_n(1,1),W_8)\geq2n+1$ for odd $n$.
	
	In order to show that $R(S_n(1,1),W_8) \leq 2n+1$ for odd $n \geq 5$, consider any graph $G$ of order $2n+1$ containing no $S_n(1,1)$ and $W_8\not\subseteq\overline{G}$. First, we will show that $G$ must contain a $S_{n-1}$. In the case of $n\geq7$, by Theorem \ref{RSnW8even} we obtain that $R(S_{n-1},W_8)=2(n-1)+2=2n$ for odd $n\geq7$. 
	Since $\overline{G}$ contains no $W_8$, then $G$ must contain a $S_{n-1}$.
	
	Now, consider if $n=5$. For a contradiction, assume that $G$ contains no $S_4$. 
	Then, it implies that $\Delta(G)\leq2$, and so $\delta(\overline{G})\geq 8$. 
	Now, consider the graph $\overline{G}$. Let $v$ be a vertex in $\overline{G}$ and 
	let $A=\{v_1,v_2,\cdots,v_8\}\subseteq N_{\overline{G}}(v)$. 
	Every vertex in $A$ has at least $8$ neighbours di $\overline{G}$. 
	Since there are only $3$ vertices in $\overline{G}$ that are not in $A$, 
	then every vertex in $\overline{G}[A]$ have at least $5$ neighbours in $\overline{G}[A]$, 
	so $\Delta(\overline{G}[A])\geq5$. 
	In \cite{BON71}, Bondy showed that if a graph $G_1$ with $n_1$ vertices have $\delta(G_1)\geq n_1/2$, then $G_1$ contains a cycle of any order less or equal to $n_1$ or  $G_1=K_{n_1/2,n_1/2}$ for even $n_1$. 
	Therefore $\overline{G}[A]$ contains a $C_8$ or $\overline{G}[A]=K_{4,4}$. Since $K_{4,4}$ contains a $C_8$, 
	then $\overline{G}$ contains a $W_8$, a contradiction. Hence, $G$ contains a $S_{n-1}$ (for $n=5$).
	
	Let $U=\{u_0, u_1,u_2,\cdots,u_{n-2}\}$ be the set of vertices of $S_{n-1}$ in $G$ with $u_0$ as the center. 
	Let $W=\{w_1,w_2,\cdots,w_{n+2}\}$ be the set of all vertices in $V(G)-U$. 
	Since $G$ contains no $S_n(1,1)$, then $E(U,W)=\emptyset$.
	Now, consider the following two cases.
	
	{\bf Case 1.} $\delta(G[W])\leq n-3$.
	Let $\deg_{G[W]}(w_1)=\delta(G[W])\leq n-3$, since $|W|=n+2$ then $\deg_{\overline{G}[W]}(w_1)\geq 4$. Let $\{w_2,w_3,w_4,w_5\}\subseteq N_{\overline{G}[W]}(w_1)$, Then,
	the induced subgraph in $\overline{G}$ by $\{u_1$, $u_2$, $u_3$, $u_4$, $w_1$, $w_2$, $w_3$, $w_4$, $w_5\}$ will contain 
	a $W_8$ with $w_1$ as the hub, a contradiction.
	
	{\bf Case 2.}  $\delta(G[W])\geq n-2$.
	Let $N_G(w_1)=\{w_2,w_3,\cdots,w_{n-1}\}$. Since $n\geq5$ and $\delta(G[W])\geq n-2$, $w_n$ has at least $3$ neighbours in $G[W]$, then $w_n$ have at least a neighbour in $N_G(w_1)$ 
	which makes a $S_n(1,1)$ in $G[W]$, a contradiction.
\end{proof}

\begin{theorem} 
	\label{Th2}
	$R(S_n(1,1),W_8)=2n$ for even $n \geq 6$.
\end{theorem}
\begin{proof} 
	Consider the graph $G=\overline{2K_4}\cup K_{n-1}$ for $n=8$ and 
	$G=\overline{C}_n\cup K_{n-1}$ for $n\ne8$. Then, $G$ contains no $S_n(1,1)$ and its complement contains no $W_8$. 
	Hence, $R(S_n(1,1),W_8) \geq 2n$.
	
	Now, to show that $R(S_n(1,1),W_8) \leq 2n$ for even $n \geq 6$, consider any graph $G$ of order $2n$ containing no $S_n(1,1)$ and $W_8\not\subseteq\overline{G}$.
	
	By Theorem \ref{RSnW8odd}, we obtain that $R(S_{n-1},W_8)=2(n-1)+1=2n-1$ for even $n$. Since $\overline{G}$ contains no $W_8$, then $G$ contains a $S_{n-1}$. 
	Let $U=\{u_0, u_1,u_2,\cdots,u_{n-2}\}$ be the set of vertices of $S_{n-1}$ in $G$ with $u_0$ as the center.
	Let $W=\{w_1,w_2,\cdots,w_{n+1}\}$ be the set of all vertices in $V(G)-U$. 
	Since $G$ contains no $S_n(1,1)$, then $E(U,W)=\emptyset$.
	Now, consider the following cases.
	
	{\bf Case 1.} $\delta(G[W])\leq n-4$. 
	Let $\deg_{G[W]}(w_1)=\delta(G[W])\leq n-4$. since $|W|=n+1$, then $\deg_{\overline{G}[W]}(w_1)\geq 4$. Let $\{w_2,w_3,w_4,w_5\}\subseteq N_{\overline{G}[W]}(w_1)$, then 
	the induced subgraph in $\overline{G}$ by $\{u_1,u_2,u_3,u_4,w_1,w_2,$ $w_3,w_4,w_5\}$ will contain 
	a $W_8$ with the hub $w_1$, a contradiction.
	
	{\bf Case 2.} $\delta(G[W])\geq n-3$. 
	If $G[W]$ has a vertex of degree at least $n-2$, say $w_1$, and 
	let $N_G(w_1)=\{w_2,w_3,\cdots,w_{n-1}\}$. 
	Since $n\geq6$ and $\delta(G[W])\geq n-3$, $w_n$ has at least $3$ neighbors in $G[W]$. 
	Hence, one of neighbors of $w_n$ is in $N_G(w_1)$ and it induces a $S_n(1,1)$ in $G[W]$, a contradiction.
	Therefore, $G[W]$ must be $(n-3)$-regular. But, this is not possible since the order of $G[W]$ is odd, a contradiction.
\end{proof}

\begin{corollary}
	\label{Cr1}
	Let $G$ be a graph of order $2n$ containing no $S_n(1,1)$ with $n \geq 6$.
	If $\overline{G}$ contains no $W_8$ then $n$ is odd and $G=G_1 \cup G_2$
	with $G_1$ is a graph of order $n-1$ and $G_2$ is 
	a regular graph of degree $n-3$.
\end{corollary}
\begin{proof}
	Let $G$ be a graph of order $2n$ satisfying the above assumption. 
	By Theorem \ref{Th2}, $n$ must be odd. 
	Theorem \ref{RSnW8even} states that $R(S_{m-1},W_8)=2m$ for odd $m$. 
	Since $\overline{G}$ contains no $W_8$ and $|V(G)|=2n$ then $G$ must contain $S_{n-1}$ by Theorem \ref{RSnW8even}.
	Now, by a similar argument as in the proof of Theorem \ref{Th2}, we obtain 
	$G=G_1 \cup G_2$
	with $G_1$ is a graph of order $n-1$ and $G_2$ is 
	a regular graph of degree $n-3$. 
\end{proof}

\begin{theorem}
	\label{Th3}
	$R(S_n(1,2),W_8)=R(S_n(2,1),W_8)=R(S_n(3),W_8)=2n$ for even $n \geq 8$.
\end{theorem}

\begin{proof}
	Consider $G=\overline{(\frac{n}{4})K_4}\cup K_{n-1}$ for $n\equiv 0 \;(\text{mod}\ 4)$ and $G=\overline{K_{3,3}\cup(\frac{n-6}{4})K_4}\cup K_{n-1}$ for $n\equiv 2(\text{mod}\ 4)$. Then, $G$ contains no tree $T$ with $\Delta(T)\geq n-3$ and its complement contains no $W_8$.  Hence, we have $R(S_n(1,2),W_8)\geq2n$, $R(S_n(2,1),W_8)\geq2n$, and $R(S_n(3),W_8)\geq2n$ for even $n$.
	
	Now, for even $n\geq8$, let $G$ be a graph of order $2n$ and assume $\overline{G}$ contains no $W_8$. By Theorem \ref{Th2}, $G$ contains $S_n(1,1)$. By Lemma \ref{Lm2}, $G$ contains $S_n(1,2)$, $S_n(2,1)$, and $S_n(3)$. Hence $R(S_n(1,2),W_8)=R(S_n(2,1),W_8)=R(S_n(3),W_8)=2n$ for even $n \geq 8$.
\end{proof}

\begin{theorem}
	\label{Th4}
	$R(S_n(1,2),W_8)=2n+1$ for $n \geq 11$ and $n\equiv 3\;(\text{mod}\ 4)$.
\end{theorem} 

\begin{proof}
	If $G=\overline{(\frac{n+1}{4})K_4}\cup K_{n-1},$ then $G$ contains no $S_n(1,2)$ and its complement contains no $W_8$. Hence, we have $R(S_n(1,2),W_8)\geq2n+1$ for $n\equiv 3 \;(\text{mod}\ 4)$. Now, for any $n\geq7$ and $n\equiv 3 \;(\text{mod}\ 4)$, 
	let $G$ be a graph of order $2n+1$ and assume $\overline{G}$ contains no $W_8$. 
	By Theorem \ref{Th1}, $G$ contains a $S_n(1,1)$. By
	Lemma 2, $G$ contains $S_n(1,2)$. Hence, $R(S_n(1,2),W_8)=2n+1$ for $n \geq 7$ and $n\equiv 3 \;(\text{mod}\ 4)$.
\end{proof}

\begin{theorem}
	\label{Th5}
	$R(S_n(1,2),W_8)=2n$ for $n \geq 9$ and $n\equiv 1 \;(\text{mod}\ 4)$.
\end{theorem}

\begin{proof}
	Let $G=\overline{3C_3\cup (\frac{n-9}{4})K_4}\cup K_{n-1}.$ Then, $G$ contains no $S_n(1,2)$ and its complement has no $W_8$. Hence, we have $R(S_n(1,2),W_8)\geq2n$ 
	for $n\equiv 1 \;(\text{mod}\ 4)$. Now, for any $n\geq9$ and $n\equiv 1 \;(\text{mod}\ 4)$, 
	assume for a contradiction $G$ as a graph of order $2n$ containing no $S_n(1,2)$ and $\overline{G}$ contains no $W_8$.
	
	If $G$ contains $S_n(1,1)$, then by Lemma \ref{Lm2}, $G$ contains a $S_n(1,2)$, a contradiction. Therefore $G$ contains no $S_n(1,1)$. By Corollary \ref{Cr1}, we have $G= G_1 \cup G_2$, 
	where $G_1$ is a graph of order $n-1$ and $G_2$ is a $(n-3)$-regular graph of order $n+1$. 
	Let $V(W)=\{w_1,w_2,\cdots,w_{n+1}\}$ be the set of vertices of $G_2$. 
	Let $N(w_1)=\{w_2,w_3,\cdots,w_{n-2}\}$ and $w_{n-1},w_n,w_{n+1}\notin N(w_1)$. For any $i \in \{n-1,n,n+1\}$, $N(w_i)\cap N(w_1)\ne \emptyset$ because $n\geq9$. 
	If $(w_i,w_j) \in E(G)$ for any $i,j \in \{n-1,n,n+1\}$ then $G[W]$ contains a $S_n(1,2)$. Therefore, the subgraph induced by $\{w_1,w_{n-1},w_n,w_{n+1}\}$ is isomorphic to 
	$\overline{K_4}$. 
	That means $|W|$ must be a multiple of 4. But, this is not possible since 
	$|W|=n+1\equiv 2(\text{mod}\ 4)$.
\end{proof}

\begin{theorem}
	\label{Th6}
	$R(S_n(2,1),W_8)=R(S_n(3),W_8)=2n-1$ for odd $n \geq 9$.
\end{theorem}

\begin{proof}
	Consider $G=2K_{n-1}$. Then, $G$ contains no tree of order $n$ and its complement contains no $W_8$. Hence we have, $R(S_n(2,1),W_8)\geq 2n-1$ and $R(S_n(3),W_8)\geq 2n-1$. Now for any odd integer $n\geq9$, let $G$ be a graph of order $2n-1$ and assume 
	$\overline{G}$ contains no $W_8$. We will prove that $G$ contains a $S_n(2,1)$ and a $S_n(3)$. From Theorem \ref{RSnW8odd}, $G$ contains a $S_{n-2}$. Now, consider the following cases.
	
	{\bf Case 1.} $\Delta(G)\geq n-2$.
	Let $u_0$ be a vertex with degree $\Delta(G)\geq n-2$. Let  $U=\{u_1,u_2,\cdots,u_{n-2}\}\subseteq N(u_0)$, and $W=V(G)-(U\ \cup\{u_0\})=\{w_1,w_2,\cdots,w_n\}$.
	
	\textit{Subcase 1.1.} $E(U,W)=\emptyset$.
	If $\delta(G[W])\leq n-5$, then $\Delta(\overline{G}[W])\geq 4$. Let $w_1$ be a vertex with degree $\delta(G[W])\leq n-5$ and $w_2,w_3,w_4,w_5\notin N(w_1)$. Then, 
	there will be a $W_8$ in $\overline{G}$ with $w_1$ as its hub and the vertices
	$u_1,u_2,u_3,u_4,w_2,w_3,w_4,$ and $w_5$ as its cycle, a contradiction.
	
	Therefore $\delta(G[W])\geq n-4$. However, $G[W]$ cannot be $(n-4)$-regular since $n$ is odd. Therefore, there is a vertex of degree at least $n-3$ in $G[W]$. 
	Therefore $G[W]$ contains a $S_n(3)$ and $S_n(2,1)$ by Lemma \ref{Lm3}.
	
	\textit{Subcase 1.2.} $E(U,W)\ne\emptyset$. Let $u_1w_1\in E(G)$, $U'=U-\{u_1\}$, and $W'=W-\{w_1\}$.
	First, assume $G$ contains no $S_n(2,1)$. Since $G$ contains no $S_n(2,1)$, then $E(G[U'])=\emptyset$ and $E(U',W')=\emptyset$. If $\delta(G[W])\geq n-3$, then $G[W]$ contains a $S_n(2,1)$ by Lemma \ref{Lm1}, so $\delta(G[W])\leq n-4$. Now, consider the following subcases: \\
	(a) If $|N_U(w_1)|\geq3$, then $N(u_1)=\{u_0,w_1\}$, since otherwise $G$ contains a $S_n(2,1)$. Since $E(U',W')=\emptyset$, then we will have a $W_8$ in $\overline{G}$ with
	the hub $u_1$ and the cycle formed by the vertices $u_2,u_3,u_4,u_5,w_2,w_3, 
	w_4,w_5$, a contradiction;\\
	(b) If $|N_U(w_1)|=2$, let $N_U(w_1)=\{u_1,u_{n-2}\}$. Then $N(u_1)\subseteq\{u_0,u_{n-2},w_1\}$, since otherwise $G$ contains a $S_n(2,1)$. 
	Since $E(U',W')=\emptyset$, then $\overline{G}$ will contains a $W_8$ with $w_1$ as the hub and $u_2,u_3,u_4,u_5,w_2,w_3,w_4,w_5$ forms its cycle, a contradiction;\\
	(c) If $N_U(w_1)=\{u_1\}$. Let $w$ be a vertex in $G[W]$ with degree $\delta(G[W])\leq n-4$ and $w^1,w^2,w^3\notin N_{G[W]}(w)$. Then, again in $\overline{G}$ we will have a $W_8$ formed by the vertices $w,w^1,w^2,w^3,u_2,u_3,u_4,u_5,u_6$ with $w$ as the hub, a contradiction. Therefore, in any case $G$ will contain a $S_n(2,1)$.
	
	Next, assume $G$ contains no $S_n(3)$, then $N(u_1)\subseteq\{u_0,w_1\}$ and $|N_W(u_i)|\leq 1$ for $i=2,3,4,\cdots,n-2$. Since $E(U',W')\leq n-3$ and $|W|=n\geq9$, there exist four vertices $w_2,w_3,w_4,w_5$ such that $N_U(w_i)\leq1$. 
	This implies that  
	$\overline{G}$ contains a $W_8$ fomed by $u_1,u_2,u_3,u_4,u_5,w_2,w_3,w_4,w_5$ with $u_1$ as the hub, a contradiction. Therefore, $G$ must contain a $S_n(3)$.\\
	
	{\bf Case 2.} $\Delta(G)=n-3$.
	Let $u_0$ be a vertex with degree $n-3$. Let  $U=N(u_0)=\{u_1,u_2,\cdots,u_{n-3}\}$, and $W=V(G)-(U\ \cup\{u_0\})=\{w_1,w_2,\cdots,w_{n+1}\}$.
	
	\textit{Subcase 2.1.} $E(U,W)=\emptyset$.
	If $\delta(G[W])\leq n-4$, then $\Delta(\overline{G}[W])\geq 4$. 
	Let $w_1$ be a vertex with degree $\Delta(\overline{G}[W])$ and $w_2,w_3,w_4,w_5\in N_{\overline{G}[W]}(w_1)$. 
	Since $E(U,W)=\emptyset$, then the vertices 
	$u_1,u_2,u_3,u_4,w_1,w_2,w_3,w_4,w_5$ will form a $W_8$ in $\overline{G}$ with
	the hub $w_1$, a contradiction. 
	Therefore $\delta(G[W])\geq n-3$. Let $w_2,w_3,\cdots w_{n-2}\in N(w_1)$, 
	then $G[W-\{w_1\}]$ is a graph of order $n$ with $\delta(G[W-\{w_1\}])\geq n-4$. The graph
	$G[W-\{w_1\}]$ cannot be $(n-4)$-regular since $n$ is odd. Therefore, 
	$\Delta(G[W-\{w_1\}]) = n-3$. By Lemma \ref{Lm3}, $G[W-\{w_1\}]$ contains a $S_n(2,1)$ and a $S_n(3)$.
	
	\textit{Subcase 2.2.} $E(U,W)\ne\emptyset$.
	Let $u_1w_1\in E(G)$, $U'=U-\{u_1\}$, and $W'=W-\{w_1\}$.
	If $\delta(G[W'])\geq n-4$ then $G[W']$ cannot be $(n-4)$-regular since $n$ is odd.
	Therefore, $\Delta(G[W'])= n-3$. But, by Lemma \ref{Lm3} $G[W']$ contains a $S_n(2,1)$ and a $S_n(3)$. Thus, $\delta(G[W'])\leq n-5.$
	
	Now, we will prove $G$ contains $S_n(2,1)$. 
	Assume to the contrary, $G$ contains no $S_n(2,1)$. Since $E(U,W)\ne\emptyset$, let $u_1w_1\in E(G)$, $U'=U-\{u_1\}$, and $W'=W-\{w_1\}$. Since $G$ contains no $S_n(2,1)$, then $E(U',W')=\emptyset$. Since $\delta(G[W'])\leq n-5$, we have that 
	$\Delta(\overline{G}[W'])\geq 4$. Let $w_2$ be a vertex of degree $\delta(G[W'])$ and $w_3,w_4,w_5,w_6\in N_{\overline{G}[W']}(w_2)$, then the vertices 
	$u_2,u_3,u_4,u_5,w_2,w_3,w_4,w_5,w_6$ will form a $W_8$ in $\overline{G}$ with the hub $w_2$, a contradiction, and hence $G$ contains $S_n(2,1)$.
	
	Next, we prove that $G$ must contain $S_n(3)$. 
	Assume to the contrary that $G$ contains no $S_n(3)$. 
	Let $a=\max\{|N_U(w_i)|:1\leq i\leq n+1\}$. Let $w_{n+1}$ be a vertex in $W$ with $a$
	neighbours in $U$. Then, by a similar argument above, 
	$W_1=W-\{w_{n+1}\}$ has a vertex $w_1$ with degree $n-5$. 
	Let $w_2,w_3,w_4,w_5\notin N_{G[W_1]}(w_1)$. Consider the following cases.\\
	(a) $a \geq3$. Let $u_1,u_2,u_3 \in N_U(w_{n+1})$. 
	Since $G$ contains no $S_n(3)$ and $N(u_0)=U$, then 
	there will be no edges between $\{u_0,u_1,u_2,u_3\}$ and $\{w_1,w_2,w_3,w_4,w_5\}$.
	Therefore, the vertices $u_0$, $u_1$, $u_2$, $u_3$, $w_1$, $w_2$, $w_3$, $w_4$, $w_5$ will form a $W_8$ 
	$\overline{G}$ with the hub $w_1$ and $u_1w_2u_2w_3u_3w_4u_0w_5u_1$ as its cycle, a contradiction.\\
	(b) $a=2$. Let $u_1,u_2\in N_U(w_{n+1})$. Since $G$ contains no $S_n(3)$, then 
	there will be no edges between $\{u_1,u_2\}$ and $\{w_1,w_2,w_3,w_4,w_5\}$. 
	We could assume $N_U(w_1)\subset\{u_{n-4},u_{n-3}\}$ since $N_U(w_1)\leq a=2$. 
	Since $n\geq 9$, we could have $u_3\notin N_U(w_1)$. 
	Since $G$ contains no $S_n(3)$, let $N_{W_1}(u_3)\subseteq\{w_2\}$. 
	Therefore, the vertices $u_1,u_2,u_3,u_4,w_1,w_2,w_3,w_4,w_5$ will form a $W_8$ in 
	$\overline{G}$ with
	the hub $w_1$ and $u_1w_2u_2w_3u_3w_4u_0w_5u_1$ as its cycle, a contradiction. \\
	(c) $a=1$. Let $N_U(w_1)\subseteq\{u_{n-3}\}$, then $u_2,u_3,u_4,u_5\notin N_U(w_1)$. Since $a=1$, let $N_U(w_i)\subseteq\{u_i\}$ for $i=2,3,4,5$. Therefore 
	the vertices $u_2,u_3,u_4,u_5,w_1,w_2,w_3,w_4,w_5$ will form a $W_8$ in 
	$\overline{G}$ with the hub $w_1$ and $u_2w_3u_5w_4u_3w_2u_4w_5u_2$ as its cycle, 
	a contradiction.
\end{proof}
\begin{center}
	\subsection*{Acknowledgment}
\end{center}
This research was partially supported under WCU Program managed by Institut Teknologi Bandung, Ministry of Research, Technology and Higher Education, Indonesia.


\begin{thebibliography}{99}
\footnotesize

\bibitem{BON71}J.A. Bondy, Pancyclic graphs, \textit{J. Combin. Theory, Ser. B}, 80-84, (1971).

\bibitem{SB01} Surahmat and E.T. Baskoro, On the Ramsey Number of a path or a star versus $W_4$ or $W_5$, \textit{Proceeding of the $12^\text{th}$ Australian Workshop on Combinatorial Algorithms}, Bandung, Indonesia, July 14-17 (2001), 174--179.

\bibitem{CZZ04a} Y. J. Chen, Y. Q. Zhang, and K.M. Zhang, The Ramsey numbers of stars versus wheels, \textit{European Journal of Combinatorics} {\bf 25} (2004), 1067--1075.

\bibitem{HBA05} Hasmawati, E.T. Baskoro, H. Assiyatun, Star-Wheel Ramsey numbers, {\em J. Combin. Math. Combin. Comput.} {\bf 55} (2005), 123--128.

\bibitem{ZCZ08} Y.Q. Zhang, Y.J. Chen, K. Zhang, The Ramsey number for stars of even order versus a wheel of order nine, \textit{European Journal of Combinatorics}, (2008), 1744--1754.

\bibitem{ZCC12} Y.Q. Zhang, T.C.E. Cheng, Y.J. Chen, The Ramsey number for stars of odd order versus a wheel of order nine, \textit{Discrete Mathematics Algorithms and Applications}, (2012).

\bibitem{BSM02} E.T. Baskoro, Surahmat, S.M. Nababan, and M. Miller, 
On Ramsey number for tree versus wheel of five and six vertices, \textit{Graphs and Combin.} {\bf 18}  (2002), 717--721.

\bibitem{CZZ04} Y.J. Chen, Y.Q. Zhang, and K.M. Zhang, The Ramsey number $R(T_n,W_6)$ for $\Delta(T_n)\geq n-3$, \textit{Applied Mathematics Letters}, (2004), 281--285.

\bibitem{CZZ05a} Y.J. Chen, Y.Q. Zhang, and K.M. Zhang, The Ramsey number $R(T_n,W_6)$ for small $n$, \textit{Utilitas Mathematica}, 67 (2005).

\bibitem{CZZ05b} Y.J. Chen, Y.Q. Zhang, and K.M. Zhang, The Ramsey number $R(T_n,W_6)$ for $Tn$ without certain deletable sets, \textit{Journal of System Science and Complexity}, 18 (2005) 95-101.

\bibitem{CZZ06} Y.J. Chen, Y.Q. Zhang, and K.M. Zhang, The Ramsey numbers of trees versus $W_6$ or $W_7$, \textit{European Journal of Combinatorics}, 27 (2006) 558-564.

\bibitem{SBB02} Surahmat, E. T. Baskoro, H.J. Broersma, The Ramsey numbers of large stars-like trees versus large odd wheels, {\em J. Combin. Math. Combin. Comput.} {\bf 65} (2008), 153--162.

\bibitem{ZBC16} Y. Zhang, H. Broersma, Yaojun Chen, On fan-wheel and tree-wheel Ramsey numbers, \textit{Discrete Mathematics}, (2016), 2284--2287.
\end{thebibliography}
\end{document}